\documentclass[11pt,oneside]{article}

\usepackage{enumitem}
\usepackage{amscd,latexsym,amsmath,amsthm,amssymb} 
\usepackage[dvips]{color}
\usepackage{titlesec}
\usepackage{geometry} 
\usepackage{authblk} 
\usepackage{enumitem}

\title{Invariant metrizability and projective metrizability 
\\
on Lie groups and homogeneous spaces}

\author{Ioan Bucataru, Tam\'as Milkovszki and Zolt\'an Muzsnay} 

\usepackage{graphicx}

\def\blue#1{\textcolor[rgb]{0.0,0.0,1.0}{#1}}

 \let\oldmarginpar\marginpar
\renewcommand\marginpar[1]{\oldmarginpar[\raggedleft\footnotesize #1]%
  {\blue{\raggedright \footnotesize \fbox{
      \begin{minipage}{1.0\linewidth}
        #1
      \end{minipage}
}}}}

\def\Ad{\rm Ad}
\def\Ad{\rm Ad}

\def\h{\mathfrak{h}}
\def\l{\ell}

\def\R{\mathbb R}
\def\p#1{\frac{\partial}{\partial #1}}
\def\pp#1#2{\frac{\partial#1}{\partial #2}}
\def\g{\mathfrak{g}}
\def\X#1{\mathfrak{X}(#1)}

\numberwithin{equation}{section} 
\numberwithin{figure}{section} 

\theoremstyle{plain}
\newtheorem{theorem}{Theorem}[section] 
\newtheorem{proposition}[theorem]{Proposition}
\newtheorem{lemma}[theorem]{Lemma}
\newtheorem{corollary}[theorem]{Corollary}
\newtheorem{definition}[theorem]{Definition}

\theoremstyle{definition}

\newtheorem{remark_num}[theorem]{Remark}

\theoremstyle{remark}
\newtheorem*{remark}{Remark}

\newtheorem*{acknowledgement*}{Acknowledgement}

\date{}

\newenvironment{packed_enumerate}{
  \begin{enumerate}[topsep=2pt, partopsep=0pt,leftmargin=15pt]
    \setlength{\itemsep}{2pt}
    \setlength{\parskip}{0pt}
    \setlength{\parsep}{0pt}
  }{\end{enumerate}}

\newenvironment{packed_2_enumerate}{
  \begin{enumerate}[topsep=2pt, partopsep=0pt,leftmargin=30pt]
    \setlength{\itemsep}{2pt}
    \setlength{\parskip}{0pt}
    \setlength{\parsep}{0pt}
  }{\end{enumerate}}

\begin{document}

\maketitle

\begin{abstract}
  In this paper we study the invariant metrizability and projective
  metrizability pro\-blems for the special case of the geodesic spray associated
  to the canonical connection of a Lie group. We prove that such canonical
  spray is projectively Finsler metrizable if and only if it is Riemann
  metrizable.  This result means that this structure is rigid in the sense that
  considering left-invariant metrics, the potentially much larger class of
  projective Finsler metrizable canonical sprays, corresponding to Lie groups, coincides
  with the class of Riemann metrizable canonical sprays. Generalisation of these results
  for geodesic orbit spaces are given.
\end{abstract}

\begin{description}
\item [2000 Mathematics Subject Classification:] 53B05, 53B40, 70H03, 70H30,
  70F17.

\item[Key words and phrases:] Euler-Lagrange equation; geodesics;
  metrizability and projective metrizability; Lie group, homogeneous space;
  geodesic orbit structure.
\end{description}

\section{Introduction}
\label{sec:introduction}

Lie groups represent a well developed theory of continuous symmetries of
mathematical structures, and it is an indispensable tool for modern
theoretical physics.  The algebraic and differential structures allow us to
consider a class of natural objects which have been extensively investigated
for over 100 years.  Among these, one of the most interesting, from
the differential geometric point of view, is the canonical geodesic structure. It
consists of the family of curves given by the 1-dimensional subgroups of a
Lie group and their left translated images.  The quadratic second order differential
equation (SODE) associated to this geodesic structure is called \emph{the
  canonical SODE} and the vector field corresponding to the geodesic
flow is called the \emph{canonical spray} of a Lie group.

In this paper we investigate the Riemann and Finsler metrizability and
projective metrizability of the canonical spray of a Lie group.  A spray is
called Riemann (resp.~Finsler) metrizable, if there exists a Riemann
(resp.~Finsler) metric such that its geodesics are the geodesics of the spray.
For the more general projective metrizability problem, one seeks a Riemann
(resp.~Finsler) metric whose geodesics coincide with the geodesics of the
spray, up to an orientation preserving reparameterization.  Recently, several
papers have appeared on the metrizability and projective metrizability
problems \cite{BDE, BM_2013_2, CMS, Matveev}. These articles show that the
metrizability and projective metrizability problems are very complex, and even
in low dimensional cases, the complete classification is very
difficult. Concerning Lie groups, G.~Thompson and his co-workers investigated
the inverse problem of Lagrangian dynamics for the canonical spray in a series
of papers \cite{GHT, GTM, TM, ST, Tho}. The problem
of the existence of a left-invariant variational principle for the canonical
spray was considered in \cite{Mestdag, MZ_ivp}, and for invariant second-order
differential equations, using the Helmholtz conditions, in \cite{CM}.

The notion of a Finsler metric is a generalisation of the
notion of a Riemannian metric, as Shiing-Shen Chern says ``Finsler geometry is just
Riemannian geometry without the quadratic restriction'',
\cite{Chern96}.  Therefore, every Riemann metrizable spray (necessarily
quadratic) is trivially Finsler metrizable. The converse, in general,
is not true; there are Finsler metrizable sprays (necessarily
non-quadratic) which are not Riemann metrizable. Even so, in the category of
quadratic sprays, Szab\'o's theorem \cite{SzaboZ} states that the two
notions of Finsler metrizability and Riemann metrizability coincide. The projective Finsler metrizability,
however, is \emph{essentially different} from projective Riemann metrizability, even in the case of quadratic
sprays. This follows since a quadratic spray can be projectively
equivalent, using a non-linear projective factor, to a
non-quadratic one and hence cannot be Riemann metrizable. Therefore, the
category of projective Finsler metrizable sprays is generally strictly larger
then the category of Riemann metrizable sprays, even for quadratic
sprays. 

The goal of this paper is to investigate the relationship between invariant
metrizability, and the invariant projective metrizability of the canonical
spray of a Lie group.  In the case of the invariant
metrizability problem, we ask if there exists a \emph{left-invariant} Riemann
(resp.~Finsler) metric, such that its geodesics are the geodesics of
the canonical spray. In the case of the invariant projective metrizability
problem, we ask if there exists a left-invariant Riemann (resp.~Finsler)
metric, such that its geodesics are projectively equivalent to the geodesics
of the canonical spray.  We prove that the canonical connection of a Lie group
is invariant projective Finsler metrizable if and only if it is invariant
Riemann metrizable.  This result shows that the structure is rigid in the
sense that by considering left-invariant metrics, the potentially much larger
class of projective Finsler metrizable canonical sprays, corresponding to Lie groups,
coincides with the class of Riemann metrizable canonical sprays.

We consider also homogeneous spaces $G/H$ with a special geodesic structure.  A
left invariant geodesic structure on $G/H$ is called \emph{geodesic orbit
  structure (g.o.~structure)}, if the geodesics can be derived as orbits of
1-parameter subgroups of $G$. In V.I.~Arnold's terminology these curves are
called "relative equilibria" \cite{Arnold}.  We prove that a g.o.~structure
is projectively invariant Riemann (resp.~Finsler) metrizable if and only if it
is invariant Riemann (resp.~Finsler) metrizable.  In the quadratic case we
also obtain the rigidity property: the class of projective Finsler metrizable
and Riemann metrizable g.o.~sprays coincide.

\section{Differential geometric background of the inverse problem of the
  calculus of variations}
\label{sec:2}

In this section we present the basic objects and tools required for our
investigation.  More details can be found in \cite{BM_2013_2, grif_muzs_2}.

\subsubsection*{SODEs, sprays and associated connection}

Let $M$ be a smooth, finite dimensional manifold, $TM$ its tangent bundle and $\pi$ the natural
projection. The map $J:TTM \to TTM$ denotes the canonical vertical
endomorphism and $C \in \mathfrak X (TM)$ the Liouville vector field.  If
$x=(x^{i})$ is a local coordinate system on $M$ and $(x,y)=(x^{i},y^{i})$ is
the induced coordinate system on $TM$, we then have that $J = dx^i \otimes \p
{y^i}$, and $C = y^i \p{y^i }$.  Using the vector field $C$ and Euler's
theorem on homogeneous functions, a Lagrangian $L:TM\to\mathbb R$ is an
$\l$-homogeneous function in the $y$ variable if and only if
\begin{equation}
  \label{eq:5}
  CF = \l \! \cdot\! F.
\end{equation}
A \textit{spray} is a vector field $S$ on $TM\!\setminus\! \{0\}$ satisfying
the relations $JS = C$ and $[C,S]=S$. The coordinate representation of a spray
$S$ takes the form
\begin{equation}
  \label{eq:S}
  S = y ^i \frac {\partial}{\partial x^i } +f^i (x,y)
  \frac {\partial}{\partial y^i },
\end{equation}
where the functions $f^i(x,y)$ are homogeneous of degree 2 in $y$.  $S$ is
called quadratic, if $f^i(x,y)$ are quadratic in $y$.  The \textit{geodesics}
of a spray $S$ are the curves $\gamma : I \to M$ such that $\dot \gamma$ is an
integral curve of $S$, that is $S \circ \dot \gamma = \ddot \gamma$. The curve
$\gamma$ is a geodesic of (\ref{eq:S}) if and only if it is a solution of the
SODE: $\ddot{x}^i=f^i(x, \dot{x})$, $i=1,\dots,n$.  Two family of curves are
projectively equivalent, if they coincide up to an orientation preserving
reparameterization.  In this spirit we also call two SODEs (or sprays)
projectively equivalent if their solutions (as parametrised curves) are
projectively equivalent. It is easy to show that two sprays $S$ and $\bar S$
are projectively equivalent if and only if there exists a 1-homogeneous
function $\mathcal P=\mathcal P(x,y)$ such that $\bar S = S -2 \mathcal P C$.

A \textit{non-linear connection} on $M$ is a type (1-1) tensor field
$\Gamma$ on $TM$ such that $J \Gamma = J$ and $\Gamma J = - J $. If $\Gamma $
is a non-linear connection, then $\Gamma^2 =id_{TTM}$ and the eigenspace corresponding to
the eigenvalue $-1$ is the vertical space.  The eigenspace $H$ corresponding
to the eigenvalue $+1$ is called the \emph{horizontal space}.  In the sequel
we will write $ h = \tfrac{1}{2}(I + \Gamma )$ and $v = \tfrac{1}{2}(I -
\Gamma )$, for the \label{def:horiz_vert} \textit{horizontal} and
\textit{vertical} projectors. If $S$ is a spray, then $\Gamma=[J,S]$ is a
non-linear connection which is called the \emph{canonical connection associated to $S$}.

\subsubsection*{The Euler-Lagrange ODE and PDE}

It is well known that if a \textit{Lagrangian} $L\colon TM \to \mathbb R$ is
regular, that is
\begin{math}
  \det \> \Bigl( \frac {\partial^2L}{\partial y^i \partial
    y^j } \Bigl) \ \neq 0,
\end{math}
then the $2$-form
\begin{math}
  \Omega_L : = dd_JL
\end{math}
has maximal rank. For a regular Lagrangian $L$, homogeneous of degree 2, the
vector field $S$ on $TM$ defined by the equation
\begin{equation}
  \label{E-L}
  i_S \Omega_L + d(d_CL-L)= 0
\end{equation}
(where $d_CL$ denotes the Lie derivative of $L$ with respect to $C$) is a
spray, and the geodesics of $S$ are the solutions of the Euler-Lagrange ODE.

\bigskip

\noindent
Let us fix a spray $S$ on the manifold $M$. Then, to every Lagrangian $L$, a
scalar 1-form 
\begin{math}
  \label{eq:omega}
  \omega_L \!=\! i_S \Omega_L \!+\! d(d_C L \!-\!L),
\end{math}
called the \textit{Euler-Lagrange} form can be associated. Then the
Euler-Lagrange equation can be written as
\begin{equation}
  \label{eq:omega_L}
  \omega_L = 0.
\end{equation}
The equation (\ref{eq:omega_L}) holds if and only if the geodesics of the spray
$S$ are the solutions of the Euler-Lagrange equation associated to $L$.  We
remark, that when $S$ is given, \eqref{eq:omega_L} is a system of second order
PDE on $L$ called the \emph{Euler-Lagrange PDE}.

\bigskip

\section{Metrizability and projective metrizability}

The metrizability and projective metrizability problems, for a given
spray $S$, can be formulated as follows.
\begin{definition}
  A spray $S$ is
  \begin{packed_enumerate}
  \item [1)] Riemann (resp.~Finsler) metrizable, if there exists a Riemann
    (resp. Finsler) metric whose geodesics coincide with the
    geodesics of $S$.

  \item [2)] projective Riemann (resp.~Finsler) metrizable, if there exists a
    Riemann (resp.~Finsler) metric whose geodesics are projectively
    equivalent with the geodesics of $S$.
  \end{packed_enumerate}
\end{definition}
\noindent
Both the metrizability and projective metrizability problems can be
formulated in terms of a system of partial differential equations which is
composed of the appropriate homogeneity condition and the Euler-Lagrange PDE
equations on the energy function.  
\begin{proposition}
  \label{prop:metr-proj-metr_remark}
  A spray $S$ on a manifold $M$ is
  \begin{packed_enumerate}
  \item [1)] \emph{Riemann (resp.~Finsler) metrizable} if and only if there exists
    a quadratic (resp.~homogeneous of degree $2$) function $E\colon TM \!\to\! \R$, such that the
    matrix field $(\frac{\partial^2E}{\partial y^i \partial y^j})$ is positive
    definite on $TM\!\setminus\!\{0\}$ and the equation (\ref{eq:omega_L}) is satisfied with $L:=E$.
  \item [2)] \emph{projectively Riemann (resp.~Finsler) metrizable} if and
    only if there exists a quadratic (resp.~homogeneous of degree $2$) function
    $\widehat{E}\colon TM \!\to \!\R$, such that the matrix field
    $(\frac{\partial^2\widehat{E}}{\partial y^i \partial y^j})$ is positive
    definite on $TM\!\setminus\!\{0\}$ and the equation (\ref{eq:omega_L}) is satisfied
    with $L:=\sqrt{2\widehat{E}}$.
  \end{packed_enumerate}
\end{proposition}
\begin{proof}
  For 1)~we remark that the Riemann (resp.~Finsler) metric
  $g=g_{ij}dx^i\otimes dx^j$ exists if and only if the associated energy
  function $E=g_{ij}y^iy^j/2$ exists. Thus a function $E$ is the Riemann
  (resp.~Finsler) energy function corresponding to $S$ if and only if it
  satisfies the conditions formulated in point 1)~of the proposition.

  For 2)~we note that the spray $S$ is projectively Riemann (resp. Finsler)
  metrizable if and only if there exists a Riemann (resp. Finsler) metrizable
  spray $\widehat{S}$ which is projectively equivalent to $S$.  In that case
  there exists a function $\mathcal P$, homogeneous of degree $1$, such that
  $\widehat{S}=S-2\mathcal P C$.  Let us denote by
  $\widehat{E}=\widehat{g}_{ij}y^iy^j/2$ the energy function of $\widehat{S}$
  and let $\widehat{F}=\sqrt{2 \widehat{E}}$ be the associated Finsler
  function.  It is well known that $\widehat{F}$ is invariant with respect to
  the parallel translation associated to the canonical non-linear connection $\widehat{\Gamma}$ and
  therefore we have
  \begin{math}
    d_{\widehat{h}}\widehat{F}=0
  \end{math}
  where $\widehat{h}=h\!-\!\mathcal P J\!-\!d_J\mathcal P \otimes C$ (see
  \cite{BM_2012}, section 4).  Using the homogeneity of degree $1$ of
  the $\widehat{F}$, we get
  \begin{equation}
    \label{eq:lambda}
    0 = d_{\widehat{h}}\widehat{F} = d_{h}\widehat{F}- d_{J}(\mathcal P
    \widehat{F}).
  \end{equation}
  Substituting $S$ into the above formula and using the homogeneity of
  degree $1$ of $\mathcal P$ and $\widehat{F}$ we get
  \begin{math}
    d_{\widehat{h}}\widehat{F}(S)=S \widehat{F}-2\mathcal P \widehat{F} =0.
  \end{math}
  It follows that the projective factor is given by 
  \begin{equation}
    \label{eq:9}
    \mathcal P =\tfrac{1}{2\widehat{F}}S \widehat{F}.
  \end{equation}
  Replacing $\mathcal P$ in formula \eqref{eq:lambda} and using the Fr\"olicher-Nijenhuis
  formalism we get
  \begin{equation}
    \label{eq:1}
    0=2 d_h\widehat{F}-d_J\big(\tfrac{1}{\widehat{F}} d_S\widehat{F} \,
    \widehat{F}\big) =d_{\Gamma+I}\widehat{F}-d_J d_S\widehat{F}
    =d_{[J,S]}\widehat{F}+d\widehat{F}-d_J d_S\widehat{F}
    =-i_Sdd_J\widehat{F}.
  \end{equation}
  On the other hand, since $\widehat{F}$ is a function homogeneous of
  degree 1, we have
  $C\widehat{F}=\widehat{F}$ and $dd_C\widehat{F}-d\widehat{F}=0$.
  Therefore, the Euler-Lagrange form is
  \begin{math}
    \omega_{\widehat{F}} \!=\!i_S dd_J\widehat{F} \!+\! dd_C \widehat{F}
    \!-\!d\widehat{F} \!=\! i_Sdd_J\widehat{F}.
  \end{math}
  Comparing this with \eqref{eq:1} we get that $\omega_{\widehat{F}}=0$. This
  shows that the conditions of 2) are necessary conditions.

  Conversely, let us suppose that the function $\widehat{E}$ satisfies the
  conditions of 2). If we define the projective factor $\mathcal P$ by
  using formula \eqref{eq:9}, then the spray $\widehat{S}=S-2\mathcal
  P C$ is projectively equivalent to the spray $S$. It is not  difficult to show that the spray associated to $\widehat{E}$ is given by
  $\widehat{S}$.
\end{proof}
We note, that a coordinate version of 2) was first proved by A.~Rapcs\'ak in
\cite{Rapcsak} and a coordinate free version of this the statement was given
by J.~Szilasi and Sz.~Vattam\'any in \cite{SziV}. Here we have presented a
different approach.

\begin{remark_num}
  \label{sec:rem_rap_EL}
  For a given spray $S$, the second order partial differential equation
  \begin{math}
    i_{S}dd_{J} F=0
  \end{math}
  is called the \emph{Rapcs\'ak equation}.  As we have seen in equation
  \eqref{eq:1}, for functions homogeneous of degree $1$, the Rapcs\'ak equation is
  equivalent to the Euler-Lagrange equation.
\end{remark_num}

\begin{proposition}
  \label{thm:EL-dh}
  Let $S$ be a spray and $L$ be a Lagrangian. If $L$ is a first
  integral for $S$, then we have
  \begin{equation}
    \label{eq:Lagrange_d_h}
    \omega_{L}=0 \qquad \Leftrightarrow \qquad d_hL=0
  \end{equation}
\end{proposition}
\noindent
\begin{proof}
  Consider $L$  a first integral for  $S$, that is $S(L)=d_SL=0$. Then, using
  the Fr\"olicher-Nijenhuis calculus, we get
  \begin{alignat*}{1}
    \omega_{L}& = i_Sdd_JL + dd_CL -dL = d_Sd_JL -di_Sd_J L + dd_C L- dL
    \\
    & = d_{[S,J]}L+ d_Jd_SL -d(d_{JS}-d_Ji_S+i_{[J,S]})L + dd_CL - dL
    \\
    & = - d_{\Gamma}L - dL = - d_{\Gamma+I}L = -2 d_hL,
  \end{alignat*}
which shows the equivalence of the two conditions of \eqref{eq:Lagrange_d_h}.
\end{proof}

\begin{corollary}
  \label{thm:L_Lk}
  Let $S$ be a spray, let $L$ be a non-zero first integral for 
  $S$ and $f$ a smooth non-vanishing function on ${\mathbb R}$ with
  non-vanishing derivative. Then, $L$ satisfies the Euler-Lagrange
  equation \eqref{eq:omega_L}, associated to $S$, if and only if
  $f(L)$ satisfies the Euler-Lagrange
  equation \eqref{eq:omega_L}, associated to $S$
\end{corollary}

\begin{proof}
  Using Theorem \ref{thm:EL-dh} we know that equation $\omega_{L}=0$ is
  equivalent to $d_hL=0$ and equation $\omega_{f(L)}=0$ is equivalent to
  $d_h(f(L))=0$. Moreover, since $f'(L)\neq 0$ and 
  \begin{equation}
    \label{eq:Lagrange_d_hom}
    d_h(f(L))=f'(L) \, d_hL,
  \end{equation}  
  we have $d_h(f(L))=0$ if and only if $d_hL=0$ holds.
\end{proof}
We will use Corollary \ref{thm:L_Lk} for the particular case when
$f(t)=t^k$. This general form of Corollary \ref{thm:L_Lk} corresponds
to \cite[Proposition 3.2]{AT92} and it was suggested
to us by am anonymous reviewer, to whom we express our thanks.

\section{Invariant metrizability and projective metrizability of the canonical
  flow of Lie groups}

Let $G$ be a finite dimensional Lie group.  We denote by $\lambda_{g}\colon G\to G$ the left
translation of $G$ defined by $\lambda_g(\hat g)=g \hat g$.  Let $x=(x^1,
\dots, x^n)$ be a coordinate system on $G$ and $(x,y)$ be the usual associated
standard coordinate system on $TG$ where $y=(y^1, \dots, y^n)$ with
$y^i=dx^i$.  We will be interested in investigating left invariant structures
on $TG$.  It is thus more convenient to introduce a kind of ``semi-invariant''
coordinate system using the left trivialisation $TG\cong G \times \g$. Indeed,
for every $g \in G$, the tangent space $T_{g}G$ is isomorphic to $\g$ by the
tangent map of the left translation $(\lambda^{-1}_{g})_{_*}: T_{g}G \to
\g=T_eG$.  Therefore, one can introduce a left invariant $\g$-valued
differential form $\theta\!:\!TG \to \g$, known as the Maurer-Cartan form,
defined by 
\begin{math}
  \theta \!= \!(\lambda^{-1}_{g})_{_{\!*}}\!dg.
\end{math} 
The corresponding \emph{semi-invariant coordinate system} is given by $(x^i,
\alpha^i)$ where
\begin{math}
  \alpha^i\!= \!(\lambda^{-1}_{x,*})_{_j}^idx^j.
\end{math}
The left invariant coordinate system $(x,\alpha)$ induces coordinates on the
second tangent bundle $TTG$ which will be denoted by
\begin{displaymath}
  (x^i,\alpha^i, X^i, A^i)\simeq X^i \p {x^i}\Bigl|_{(x,\alpha)} + A^i \p
  {\alpha^i}\Bigl|_{(x,\alpha)},
\end{displaymath}
We note that, by using a simplified notation
\begin{equation}
  \label{eq:notation}
  (x^i,\alpha^i, X^i, A^i) \simeq (x,\alpha, X, A),
\end{equation}
the coordinates $x\!=\!(x^i)$, and therefore the coordinates $X\!=\!(X^i)$ are
not left invariant, but $\alpha\!=\!(\alpha^i)$ and therefore the
corresponding $A\!=\!(A^i)$, are: the left translation by a group element $g$
induces on $TTG$ the following action:
$(\lambda_g)_{**}(x,\alpha, X, A) = (\lambda_gx,\alpha, \lambda_{g*}X, A
  )$.
\begin{definition}
  \label{thm:left_inv}
  A vector field $X\in \X G$ is called left invariant, if
  $(\lambda_{g})_{_*}X=X$ for every $g\in G$.  Similarly, a function $L\colon
  TG\to \R$ is called left invariant, if $L \circ \lambda_{g*}=L$, $\forall
  g\in G$.
\end{definition}
\noindent
Using the semi-invariant coordinate system $(x^i, \alpha^i)$, the function
$L\colon TG\to \R$ is left invariant if and only if its value does not depend
on the $x$-coordinates, that is
\begin{equation}
  \label{eq:inv_L}
  \frac{\partial L}{\partial x^i}\equiv 0, \qquad i=1, \dots , n.
\end{equation}

\begin{definition}
  The canonical geodesic structure on a Lie group $G$ is given by the
  1-parameter subgroups and their left (or right) translated images.  The
  canonical SODE of $G$ is the SODE corresponding to the canonical flow.
\end{definition}
\noindent
The main geometric objects associated to the canonical flow (spray, horizontal
and vertical projections etc.) were calculated in \cite{MZ_ivp} and \cite{TM}. Here
we just present the essential results needed for our purpose. More about the
computation of these objects can be found in the above mentioned papers.

Using a $G \to GL(n, \R)$, $x\to M_x$ matrix representation, the canonical
SODE can be described by
\begin{equation}
  \label{eq:sode}
  \ddot M_t  = \dot M_t M_t^{-1} \dot M_t,
\end{equation}
where we denote $M_t:=M_{x_t}$.  The \emph{canonical spray} of a Lie group $G$,
in the semi-invariant coordinate system $(x, \alpha)$, using the simplified
notation \eqref{eq:notation}, is given by
 \begin{equation}
   \label{eq:4}
   S_{(x,\alpha)}    =\lambda_x \alpha \, \p  {x}\Bigl|_{(x,\alpha)}.
\end{equation}
The vertical and horizontal projectors are defined as follows. For every
$(x,\alpha)\in TG$ we have
\begin{alignat}{1}
  \label{eq:6}
  v\Big(\lambda_xa \p x \Bigl|_{(x, \alpha)} + b\p \alpha\Bigl|_{(x,
    \alpha)}\Big) & = \bigl(\tfrac{1}{2}[a,\alpha]+b\bigl) \p
  \alpha\Bigl|_{(x, \alpha)},
  \\
  \label{eq:7}
  h\Big(\lambda_xa \p x \Bigl|_{(x, \alpha)} + b\p \alpha\Bigl|_{(x,
    \alpha)}\Big) & = \lambda_xa \p x \Bigl|_{(x, \alpha)} -
  \tfrac{1}{2}[a,\alpha] \p \alpha\Bigl|_{(x, \alpha)}.
\end{alignat}
 We have the following
\begin{proposition}
  \label{sec:lemma}
  (Canonical invariant Euler-Lagrange system) \newline A Lagrangian $L:TG\to
  \R$ is a left-invariant solution to the Euler-Lagrange equation associated
  to the canonical spray of the Lie group $G$, if and only if the system
  \begin{alignat}{2}
    \label{eq:met_1}
    \hphantom{i = 1, \dots , n \qquad}\frac{\partial L}{\partial x^i} &=0,& &i
    = 1, \dots , n
    \\
    \label{eq:met_3}
    [a, \alpha]^i \, \frac{\partial L}{\partial \alpha^i} &=0, \qquad & &
    \forall a \in \g,
  \end{alignat}
  is satisfied.
\end{proposition}

\begin{proof}
  The equation (\ref{eq:met_1}) expresses the left invariant property of $L$.
  Moreover, from the local expression \eqref{eq:4} of the canonical spray
  $S$ it is clear, that if the Lagrangian $L$ is left invariant, then
  it is also a first integral of $S$, which means that 
  \begin{equation}
    \label{eq:d_SL}
    d_SL=0.
  \end{equation}
  Using Proposition \ref{thm:EL-dh}, we get that the Euler-Lagrange PDE is
  satisfied if and only if $d_hL=0$. Using \eqref{eq:7} we have
  \begin{displaymath}
    d_hL = 0 \quad \Leftrightarrow \quad  
    \lambda_xa \pp L x  - \tfrac{1}{2}[a,\alpha] \pp L \alpha = 0,
    \quad a \in \g.
  \end{displaymath}
  and from the left invariant property we get the later is identically zero if
  and only if the equations \eqref{eq:met_3} are satisfied.
\end{proof}
\noindent 
The following statement holds.
\begin{proposition}
  \label{prop:metr-proj-metr}
  The canonical spray $S$ of a Lie group is left invariant projectively
  Riemann (resp.~Finsler) metrizable if and only if it is left invariant
  Riemann (resp.~Finsler) metrizable.
\end{proposition}
\begin{proof}
  It is clear that if $S$ is Riemann (resp.~Finsler) metrizable, then it is
  also projectively Riemann (resp.~Finsler) metrizable.  Conversely, let us
  suppose that $S$ is projectively Riemann (resp.~Finsler) metrizable.  Then,
  according to Proposition \ref{prop:metr-proj-metr_remark}, there exists a
  left invariant (quadratic, resp.~homogeneous of degree $2$) function $\widehat{E}:TG \to
  \R$ such that the matrix field
  \begin{math}
    (\frac{\partial^2 \widehat{E}}{\partial y^i \partial y^j})
  \end{math}
  is positive definite on $TM\setminus\{0\}$ and
  $\widehat{F}:=\sqrt{2\widehat{E}}$ satisfies the Euler-Lagrange PDE
  associated to $S$. Because of the left invariance condition, we have $d_SL=0$ and,
  using Corollary \ref{thm:L_Lk}, we get that
  $\widehat{E}:=\frac{1}{2}(\widehat{F})^2$ is also a solution of the
  Euler-Lagrange PDE associated to $S$. Then, according to Proposition
  \ref{prop:metr-proj-metr_remark}, the given $\widehat{E}$ is the energy
  function of a Riemann (resp.~Finsler) metric which implies that $S$ is
  Riemann (resp.~Finsler) metrizable.
\end{proof}
We have the following result.
\begin{theorem}
  \label{prop:metr-proj-metr-1}
  The canonical spray of a Lie group is left invariant projectively Finsler
  metrizable if and only if it is left invariant Riemann
  metrizable.
\end{theorem}
\begin{proof}
  In one direction the statement is trivial. If the canonical spray is Riemann
  metrizable, then it is trivially Finsler metrizable and also projectively
  Finsler metrizable.
  \\
  Let us consider the converse statement, and suppose that the canonical spray
  $S$ is projectively Finsler metrizable. Then, according to Proposition
  \ref{prop:metr-proj-metr}, it is also Finsler metrizable.  Since $S$ is
  quadratic, it follows that the associated connection is linear. Hence, the
  Finsler metrizability induces the existence of a Berwald metric on the Lie
  group. Using Szab\'o's theorem which states that for every Berwald metric
  there exists a Riemannian metric such that the geodesics of the Berwald and
  Riemannian metrics are the same (cf.~\cite{SzaboZ}), we get that the
  canonical spray is Riemann metrizable.
\end{proof}
\medskip Using Proposition \ref{sec:lemma} we obtain the following.
\begin{corollary}
  \label{prop:metr-proj-canon}
  The canonical spray of a Lie group $G$ is left invariant (Riemann, Finsler,
  projectively Riemann or projectively Finsler) metrizable if and only if there
  exists a scalar product $\langle \ , \ \rangle$ on $\g$ such that
  \begin{equation}
    \label{eq:EL_met}
    \langle [a,\alpha], \alpha \rangle = 0
  \end{equation}
  for every $a,\alpha \in \g$.
\end{corollary}
\begin{proof}
  An invariant Riemannian metric induces a scalar product $\langle \ , \
  \rangle$ on $\g$.  Using the coordinate system $(x,\alpha)$ on $TG\simeq
  G\times \g$, the associated energy function is given by $E: G\times \g \to
  \R$, where $E(x,\alpha)= \langle \alpha, \alpha \rangle$.  The
  Euler-Lagrange equation \eqref{eq:met_3} then implies \eqref{eq:EL_met}.
\end{proof}

\begin{remark_num}
  We want to draw attention to few interesting phenomena.  First of all,
  although the canonical spray of a Lie group is a very natural object, it is
  not true that it is always metrizable. In \cite{GHT} there are several
  examples of Lie groups and Lie algebras where the canonical spray is non
  metrizable.
  \\
  Secondly, despite the fact that the canonical spray is left (and also right)
  invariant, and the Euler-Lagrange equation inherits the symmetries of the
  Lagrangian, it is not true that the ``\emph{metrizability}'' property means
  automatically ``\emph{metrizability by a left invariant metric}''.
\end{remark_num}
\noindent Indeed, for example the 3-dimensional Heisenberg group $\mathbb H_3$
is not metrizable or projectively metrizable with an invariant Riemann (or
Finsler) metric \cite{MZ_ivp}. However, since the curvature tensor vanishes
identically, the canonical spray is metrizable.  The corresponding (non
invariant) Riemannian metric is given by 
\begin{math}
  g \!=\! dx^2 \!+\! dy^2 \!+\! (dz\!- \frac{y}{2} dx\!-\! \frac{x}{2}
    dy)^2,
\end{math}
(see \cite{GHT}).

\medskip

Theorem \ref{prop:metr-proj-metr-1} shows that the geometric structure
associated to the canonical spray of a Lie group has a certain rigidity
property. The potentially much larger class of Lie groups where the canonical
spray is projectively equivalent to a Finsler geodesic structure, actually
coincides with the class of invariant Riemann metrizable sprays.  We note that
this property relies heavily on the fact that, 1) the canonical spray is
quadratic, and 2) the Lie derivative of a left invariant Lagrange function on
$G$, with respect to the canonical spray, is identically zero. The second
property is not true in general for an arbitrary left invariant spray. However,
interesting generalisation can be obtained by considering the class of
homogeneous spaces. We consider this in the next section.

\section{Invariant metrizability and projective metrizability of a geodesic
  orbit structure of homogeneous spaces}
\label{sec:homogeneous}

Let $M$ be a connected differentiable manifold on which the Lie transformation
group $G$ acts transitively.  Let us fix an origin $o \in M$ and denote by $H$
the stabiliser of $o \in M$ in the group $G$ and by $\pi : G \to G/H$ the
projection map. As usual we call $H$ the isotropy group of the homogeneous
space $G/H$. Then $M$ is isomorphic to the factor space $G/H$ with origin $H$
and its tangent space at $o\in M$ is isomorphic to $\g / \h$, where $\g$ and
$\h$ are the Lie-algebras of the Lie groups $G$ and $H$ respectively.  The
action of $G$ on $M$ is determined by the map
\begin{displaymath}
  \lambda:(g,m)\mapsto \lambda _g m = g\cdot m: G \times M \rightarrow M. 
\end{displaymath}
Geodesic structures, sprays, metrics, Lagrangians on $M$ are called
\emph{invariant}, if they are invariant with respect to the action of $G$.  It
is clear, that invariant sprays, metrics, and Lagrangians can be characterised
by their values on $T_oM$.  

In the most interesting cases, the algebraic structures and invariant
geometric structures are intimately related: the geodesics are the image of
the 1-parameter subgroups of the group $G$ and their left translated
images. More precisely we have the following
\begin{definition}
  \label{def:hom_geod}
  A geodesic $\gamma (t)$, emanating from the origin $o\! \in\! M$, is called
  \emph{homogeneous}, if there exists $X_\gamma \!\in \!  \g$, such that
  $\gamma (t)$ is the orbit of the 1-parameter subgroup $\{\exp tX_\gamma,
  \;t\in \mathbb R\}$ of $G$, that is
  \begin{equation}
    \label{eq:21}
    \gamma (t) = \lambda _{\exp tX_\gamma } o = (\exp tX_\gamma) \cdot o.
  \end{equation}
  The Lie algebra element $X_\gamma \!\in \!  \g$ is called the \emph{geodesic
    vector} associated to the direction $\dot \gamma(0)\in T_oM$.
\end{definition}

\begin{definition}
  A left invariant geodesic structure is called \emph{geodesic orbit
    structure} (g.o.~structure), if any geodesic $\gamma (t)$ emanating from
  the origin $o\! \in\! M$ is homogeneous (in the sense of Definition
  \ref{def:hom_geod}). A spray is called \emph{geodesic orbit spray}
  (g.o.~spray), if it corresponds to a g.o.~structure.
\end{definition}
\noindent
Homogeneous geodesics are called in V.I.~Arnold's terminology ``\emph{relative
  equilibria}'' \cite{Arnold}.

\begin{definition} 
  A map $\sigma\colon T_oM \to \g$ is called a \emph{homogeneous lift}\footnote{In
    \cite{MzNp} the terminology \emph{horizontal lift} was used. However, in
    Finsler geometry, this terminology is widely used for a different object.}
  if the following conditions are satisfied:
  \begin{packed_2_enumerate}
  \item $\pi_*\circ \sigma = id_{T_oM}$.
  \item $\sigma$ is 1-homogeneous, that is
    \begin{math}
      \sigma(\kappa\cdot v) = \kappa \,\sigma (v),
      \end{math}
      for every $v \in T_o M$ and $\kappa \in \R$.
  \item $\sigma$ is ${\Ad}(H)$-invariant, that is
      \begin{math}
        \sigma({\lambda}_{h*}v) = {\Ad}_h\sigma(v)
      \end{math}
      for all $h \in H$ and $v \in T_oM$.
    \end{packed_2_enumerate}
    The homogeneous lift $\sigma$ is called $\mathcal{C}^{\infty}$-differentiable
    if it is continuous on $T_oM$ and $\mathcal{C}^{\infty}$-differentiable on
    $T_oM\setminus \{0\}$.
\end{definition}
It is clear that any g.o.~spray determines a
$\mathcal{C}^{\infty}$-differentiable homogeneous lift by associating to $v\in
T_oM$ its geodesic vector $X=\sigma(v)$ and vice versa, every homogeneous lift
determines a g.o. spray by left translations.

\begin{lemma}
  \label{thm:lemma_inv}
  Invariant functions are constant along the geodesics of a g.o.~spray.
\end{lemma}

\begin{proof}
  Because of the invariance, it is enough to show this property for geodesics
  emanating from the origin $o \in M$.  Let $S$ be a geodesic orbit spray on a
  homogeneous space $M=G/H$.  For a geodesic $\gamma(t)$ emanating from the
  origin $o \in M$ there exists a geodesic vector $X_\gamma\in \g$ such that
  we have \eqref{eq:21}. Using the 1-parameter subgroup property we can write
  \begin{equation}
    \label{eq:2}
    \gamma (t_0+t) = \lambda _{\exp (t_0+t) X_\gamma } o= \lambda _{\exp t_0
      X_\gamma } \lambda _{\exp t X_\gamma } o
  \end{equation}
  and therefore we have
  \begin{equation}
    \label{eq:3}
    \dot \gamma (t_0) =\frac{d}{dt}\Big|_{t=0} \gamma (t_0+t) 
    = \frac{d}{dt}\Big|_{t=0} \lambda _{\exp t_0 X_\gamma } \lambda _{\exp t X_\gamma } o
    =  (\lambda _{\exp t_0 X_\gamma })_* \dot \gamma(0).
  \end{equation}
  Let $L\colon TM\to\R$ be an invariant function.  For every $p\!\in\! M$,
  $v\!\in\! T_pM$ and $g\!\in \!G$ we have $L(v)=L(\lambda_{g,*}v)$. Applying
  this with $v=\dot \gamma(0)$ and $g=\exp t_0 X_\gamma$ we get from
  \eqref{eq:3} that $L(\dot \gamma (t_0))=L(\dot \gamma (0))$ that is $L$ is
  constant on the geodesics.
\end{proof}
Using Lemma \ref{thm:lemma_inv} we can obtain the following
generalisation of the Proposition \ref{prop:metr-proj-metr}.
\begin{proposition}
  A g.o.~spray is projectively Riemann (resp.~Finsler) invariant metrizable if
  and only if it is invariant Riemann (resp.~Finsler) metrizable.
\end{proposition}
\begin{proof}
  Let $S$ be a g.o.~spray.  As Lemma \ref{thm:lemma_inv} shows, if $L\colon
  TM\to\R$ is a $G$ invariant Lagrangian, then $L$ is constant along the
  geodesics, that is along the integral curves of $S$.  Consequently we have
  $d_SL=0$. Using Corollary \ref{thm:L_Lk} and the same argument that was used
  for Proposition \ref{prop:metr-proj-metr} we can obtain the proof of the
  proposition.
\end{proof}
We remark that the connection determined by a g.o.~spray is not necessarily
linear, therefore it is not true in general that the Finsler metrizability
entail the Riemann metrizability as it was the case for the canonical spray of
Lie groups.  However, if the g.o.~spray is quadratic then the associated
connection is linear. Therefore we can use Szab\'o's theorem and similarly to
Theorem \ref{prop:metr-proj-metr-1} we can get the following
\begin{theorem}
  \label{thm:metr-proj-metr-hom}
  A quadratic g.o.~spray is invariant projectively Finsler metrizable if
  and only if it is invariant Riemann metrizable.
\end{theorem}

\noindent
\begin{remark}
  A different invariant metrizability concept of the $G/H$ structure is
  considered in \cite{DengHou} by S.~Deng and Z.~Hou where the $G/H$ structure
  is called invariant metrizable if there exists an invariant metric on it.
  The invariant metrizability (and projective metrizability) of a
  g.o.~structure or g.o.~spray is however more subtle, because in this case
  not only the $G/H$ homogeneous space, but also the geodesic structure is
  fixed and we want to metrize both.  It may happen that the g.o.~structure
  on a homogeneous space $G/H$ is not invariant metrizable but the
  $G/H$ structure it is. To illustrate this phenomenon, let us
  consider the following example.
\end{remark} 
\bigskip

\noindent
\textbf{Example.} Let us consider $M=\R^2$ as the homogeneous space
$\R^2=G/H$, where $G=ASO(2)$ is the euclidean transformation group, and $H=
SO(2)$ is the rotation subgroup. $H$ is the stabiliser of the point
$o\!=\!(0,0)\!\in\!  \R^2$. The matrix representation can be given by
\begin{displaymath}
  G=
  \left\{
    \left(
      \begin{smallmatrix} 
        \cos t & -\sin t & x_1
        \\
        \sin t & \cos t & x_2
        \\
        0 & 0 & 1
      \end{smallmatrix} 
    \right) 
  \right\}, 
  \qquad 
  H=
  \left\{
    \left(
      \begin{smallmatrix} 
        \cos t & -\sin t & 0
      \\
      \sin t & \cos t & 0
      \\
      0 & 0 & 1
    \end{smallmatrix} 
  \right) 
\right\},
\end{displaymath}
and the corresponding Lie algebras are
\begin{displaymath}
  \g\!=\! \left\{ \left(
      \begin{smallmatrix} 
        0 & -a_3 & a_1
        \\
        a_3 & 0 & a_2
        \\
        0 & 0 & 0
      \end{smallmatrix} 
    \right) 
  \right\},
  \qquad 
  \h\!=\! \left\{ \left(
      \begin{smallmatrix} 
        0 & -a_3 & 0
        \\
        a_3 & 0 & 0
        \\
        0 & 0 & 0
      \end{smallmatrix} 
    \right) 
  \right\}.
\end{displaymath}
We have the identification $G/H\simeq \R^2$, and the projection map is
given by 
\begin{equation}
  \label{eq:13}
  \pi:G\to G/H\simeq \R^2,
  \qquad\quad   \left(
    \begin{smallmatrix}
      \cos t & -\sin t & x_1
      \\
      \sin t & \ \cos t & x_2
      \\
      0 & 0 & 1
    \end{smallmatrix}
  \right) \stackrel{\pi}{\to} 
  \left(
    \begin{smallmatrix}
      x_1
      \\
      x_2
      \\
      1
    \end{smallmatrix}
  \right) \sim (x_1,x_2).
\end{equation}
Using the identification ``$\sim$'' introduced in \eqref{eq:13}, the action of
$G$ can be interpreted as a matrix multiplication: $\lambda_g \widehat x =g
\cdot \widehat x$. Since the homogeneous lift for this $G/H$ structure must be
a rotation-invariant map, homogeneous of degree $1$, it is of the form
\begin{equation}
  \label{eq:10}
  \sigma_\kappa\colon T_o\R^2\to \g  , \qquad  (v_1, v_2) 
  \stackrel{\sigma_\kappa}{\longrightarrow}
  \left(
    \begin{smallmatrix} 
      0 & -\kappa \sqrt{v_1^2+v_2^2} & v_1
      \\
      \kappa \sqrt{v_1^2+v_2^2} & 0 & v_2
      \\
      0 & 0 & 0
      \end{smallmatrix} 
    \right),
\end{equation}
where $\kappa \! \in \! \R$ is an arbitrary (but fixed) constant.  The
geodesic equation corresponding to the homogeneous lift \eqref{eq:10} can be
given as follows: a curve $t\to \gamma(t)$ is a geodesic if and only if
\begin{equation}
  \label{eq:12}
  \ddot \gamma = \sigma_\kappa(\dot \gamma ) \dot \gamma.
\end{equation}
Indeed, let $\gamma(t)$ be the geodesic in the direction $v\!\in\! T_oM$ and let
$X\!=\!\sigma_\kappa(v)$ be the geodesic vector corresponding to $v$.
Using the matrix representation we get
\begin{math}
  \gamma(t)\!=\!e^{t\,\sigma_\kappa(v)} o
\end{math}
and therefore for its second derivative we can find 
\begin{math}
  \ddot \gamma(t)\!=\!\sigma_\kappa(v) \sigma_\kappa(v)
  e^{t\,\sigma_\kappa(v)} o\!=\!  \sigma_\kappa(v) \dot \gamma (t).
\end{math}
Since $v\!=\!\dot \gamma(0)$ and the tangent vectors are invariant to
translations, we can obtain \eqref{eq:12}.  The system of differential equations
of geodesics is
\begin{equation}
  \label{eq:11}
  \left\{
    \begin{aligned}
      \ddot x_1 & = -\kappa \sqrt{\dot x_1^2+\dot x_2^2} \ \dot x_2,
      \\
      \ddot x_2 & = \hphantom{-}\kappa \sqrt{\dot x_1^2+\dot x_2^2} \ \dot
      x_1.
    \end{aligned}
  \right.
\end{equation}
If $\kappa=0$, then the curvature tensor is identically zero and the geodesics are
straight lines. The corresponding g.o. structure is metrizable: the usual
euclidean metric ($g_{ij}=\delta_{ij}$) metricises it.  If $\kappa\neq 0$,
then the geodesics are non-reversible (that is $\gamma_{p,-v}(t)\neq
\gamma_{p,v}(-t)$) and the curvature tensor is non-zero:
\begin{displaymath}
  R^{\kappa}_{(x,y)} \Big(\p{x_1}, \p{x_2}\Big) = -\kappa^2 y_2 \p{y_1}+\kappa^2 y_1 \p{y_2}.
\end{displaymath}
It is easy to show, using the metrizability criteria of \cite{BM_2013_2} or
\cite{MZ_EL}, that the corresponding g.o. structure is not metrizable.
Therefore it is not invariant metrizable as well.

This examples illustrates clearly the difference between the concept of the
metrizability of the $G/H$ structure and the metrizability of the
g.o. structure: the homogeneous space $\R^2\!=\! ASO(2)/SO(2)$ is invariant
metrizable (for example an invariant metric is given by the standard euclidean
metric), however the g.o.~structure determined by the homogeneous lift
\eqref{eq:10} is not, unless $\kappa=0$.

\subsection*{Acknowledgments} We appreciate all the comments from an
anonymous reviewer, especially those related to Corollary 3.5.  This work has been supported 
by the EU FET FP7 BIOMICS project, contract number CNECT-318202.

\vspace{0.5cm}

\noindent
Ioan Bucataru
\\
Faculty of Mathematics, Alexandru Ioan Cuza University
\\
700506, Bulevardul Carol I 11, Ia\c si, Romania
\\
{\it E-mail}: {\tt {}bucataru@uaic.ro}\vspace{4mm}
\\
\noindent
Tam\'as Milkovszki
\\
Institute of Mathematics, University of Debrecen,
\\
H-4032 Debrecen, Egyetem t\'er 1, Hungary
\\
{\it E-mail}: {\tt {}milkovszki@science.unideb.hu}\vspace{4mm}
\\
\noindent
Zolt\'an Muzsnay
\\
Institute of Mathematics, University of Debrecen,
\\
H-4032 Debrecen, Egyetem t\'er 1, Hungary
\\
{\it E-mail}: {\tt {}muzsnay@science.unideb.hu}\vspace{4mm}
\\

\end{document}